\documentclass[sn-mathphys,Numbered]{sn-jnl}%

\usepackage{tikz-cd}
\usepackage{graphicx}%
\usepackage{multirow}%
\usepackage{amsmath,amssymb,amsfonts}%
\usepackage{amsthm}%
\usepackage{mathrsfs}%
\usepackage[title]{appendix}%
\usepackage{xcolor}%
\usepackage{textcomp}%
\usepackage{manyfoot}%
\usepackage{booktabs}%
\usepackage{algorithm}%
\usepackage{algorithmicx}%
\usepackage{algpseudocode}%
\usepackage{listings}%
\usepackage{float}%

\theoremstyle{thmstyleone}%
\newtheorem{theorem}{Theorem}

\newtheorem{corollary}[theorem]{Corollary}

\newtheorem{lemma}[theorem]{Lemma}

\theoremstyle{thmstyletwo}%
\newtheorem{remark}{Remark}%

\theoremstyle{thmstylethree}%
\newtheorem{definition}{Definition}%

\usepackage{natbib}
\begin{document}

\title[Article Title]{A Construction of Permutation Polynomials Using R\'{e}dei Function in Even Characteristic}

\author*[1]{\fnm{Daniel Panario} }
\author*[1]{\fnm{Nihal Uyar}}
\author*[1]{\fnm{Qiang Wang}}

\affil*[1]{\orgdiv{School of Mathematics and Statistics}, \orgname{Carleton University}, \orgaddress{\street{1125 Colonel By Dr}, \city{Ottawa}, \postcode{K1S 5B6}, \state{ON}, \country{Canada}}}

\abstract{The R\'{e}dei function defined over a field of even characteristic has been introduced by N\"{o}bauer in 1986.
In this paper, inspired by the work of Fu et al.  in odd characteristic, employing the AGW criterion, we present a recursive construction of permutation polynomials in even characteristic using the R\'{e}dei function over a field of characteristic 2.}


\maketitle

\section{Introduction}\label{sec1}

Let $q$ be a power of a prime number, $\mathbb{F}_q$ be the finite field with $q$ elements and $\mathbb{F}_q [x]$ be the ring of polynomials over $\mathbb{F}_q$.
It is well known that every function over $\mathbb{F}_q$ can be expressed as a polynomial over $\mathbb{F}_q$ and, an important class of polynomials is the one formed by bijections in $\mathbb{F}_q$.
A permutation polynomial $f \in \mathbb{F}_q[x]$ is a bijection in $\mathbb{F}_q$ into itself.
Studies on permutation polynomials have started with Hermite and Dickson \cite{dickson, hermite} and it is still an active area of research with several applications in cryptography \cite{crypto1,crypto2}, coding theory \cite{cod1}, combinatorial designs \cite{combdis} and on many other areas of mathematics and engineering.
For more information on permutation polynomials over finite fields, \cite{survey, surv, handbook}, and references there in provide an excellent survey.

Construction of permutation polynomials over a finite field of either odd characteristic or of even characteristic is an interesting hot topic.
For some constructions of permutation polynomials over finite fields of odd characteristic, the reader is referred to \cite{agw, odd1, odd2, odd3, wang, odd4, odd5, agw2}.
Some classes of permutation polynomials in a field of even characteristic presented in \cite{even1, even2,even3,even4, even5}.
More than a decade ago, Akbary, Ghioca and Wang introduced a useful criterion to study permutation polynomial on finite sets.\\

\begin{theorem}\cite{agw}\label{agw}
    Let $A$, $S$ and $\Bar{S}$ be finite sets with $|S|= |\Bar{S}|$.
    Let $f, \Bar{f}, \lambda, \Bar{\lambda}$ be maps on finite sets such that $f: A \rightarrow A$, $\Bar{f}: S \rightarrow \Bar{S}$, $\lambda: A \rightarrow S$, $\Bar{\lambda}: A \rightarrow \Bar{S}$ and $\Bar{\lambda} \circ f = \Bar{f} \circ \lambda$ :

\begin{center}
\begin{tikzcd}
A \arrow{d}{\lambda}  \arrow{r}{f}  &  A \arrow{d}{\Bar{\lambda}} \\
S \arrow{r}{\Bar{f}}  &  \Bar{S} \\
\end{tikzcd}
\end{center}

    If $\lambda$ and $\Bar{\lambda}$ are surjective, then the following are equivalent:
    \begin{itemize}
        \item $f$ is a bijection from $A$ to itself;
        \item $\Bar{f}$ is a bijection from $S$ to $\Bar{S}$ and $f$ is injective on $\lambda^{-1}(s)$ for each $s \in S$.
    \end{itemize}
\end{theorem}

This theorem can be used to derive the following criterion which has been broadly used to obtain permutation polynomials by considering the set $A= \mathbb{F}_q^*$ and the maps $\lambda =\Bar{\lambda}= x^s$ where $q-1= ds$ for some positive integers $s,d$ and a prime power $q$.\\

\begin{corollary}\cite{lee, w, agw3}
    Let $q-1 =ds$ where $d$ and $s$ are two positive integers and $q$ is a prime power.
    Then $$p(x)= x^rf(x^s)$$ is a permutation polynomial over $\mathbb{F}_q$ if and only if $\gcd(r,s)=1$ and $x^rf(x)^s$ permutes the set of $d$th roots of unity $\mu_d$.\\
    \end{corollary}

A function that we have an interest in this paper is the R\'{e}dei function.
The R\'{e}dei function in odd characteristic was introduced by R\'{e}dei \cite{gen}.
A couple of decades after that, N\"{o}bauer introduced the R\'{e}dei function in even characteristic \cite{even}.
Since then, the R\'{e}dei function has attracted the attention of researchers, in particular due to its permutation properties.
Another related function that we are interested in this article is the tangent-Chebyshev function \cite{lima}.
Trigonometric  functions over finite fields and the Chebyshev function have important applications in cryptography \cite{cry1, cry2,cry3,cry4} and in engineering applications\cite{eng1,eng2,eng3}.
The tangent-Chebyshev function in characteristic 2 defined by means of R\'{e}dei function appears recently \cite{zieve}.

In this article, we employ this corollary to present two new classes of permutation polynomials over a finite field of even characteristic $\mathbb{F}_{q^2}$ where $q=2^t$ for some positive odd integer $t$.
We present a recursive construction of permutation polynomials over $\mathbb{F}_{q^2}$ in even characteristic using the R\'{e}dei function.
We also give  tables of  examples of permutation polynomials in different fields of characteristic $2$, including the exhaustive list for $\mathbb{F}_{2^6}$. 

We give the structure of the paper as follows.
In Section 2, we introduce R\'{e}dei function in a field of characteristic 2 starting with its definition in odd characteristic and mentioning an analogous construction of permutation polynomials in odd characteristic.
In Section 3, we state and prove our results that provides a recursive construction of permutation polynomials in $\mathbb{F}_{q^2}$.
The conclusion is given in Section 4.

\section{R\'{e}dei functions}\label{sec2}

In this section, we define R\'{e}dei function (as well as tangent-Chebyshev function) in $\mathbb{F}_{q^2}$ where $q$ is odd.
Then state the  central result \cite[Theorem 1]{wang} that we adapt to obtain permutation polynomials in a field of even characteristic.

Let $q$ be an odd prime power, $n$ be a positive integer and $\alpha \in \mathbb{F}_{q^2}^*$.
The R\'{e}dei function over $\mathbb{F}_{q^2}$ is defined as follows:
$$R_n(x, \alpha) = \frac{G_n(x, \alpha)}{H_n(x, \alpha)}= \frac{\sum_{i=0}^{\lfloor n/ 2 \rfloor}{{n \choose 2i} \alpha^i x^{n-2i}}}{\sum_{i=0}^{\lfloor n/ 2 \rfloor}{{n \choose 2i +1} \alpha^i x^{n-2i -1}}} .$$

The degree-$n$ tangent-Chebyshev rational function
over $\mathbb{F}_{q^2}$ with a non-square $\alpha \in \mathbb{F}_{q^2}^*$ is defined as follows:
$$C_n(x, \alpha)= \frac{E_n(x, \alpha)}{F_n(x, \alpha)}= \frac{\sum_{i=0}^{\lfloor n/ 2 \rfloor}{{n \choose 2i +1} \alpha^i x^{2i +1}}}{\sum_{i=0}^{\lfloor n/ 2 \rfloor}{{n \choose 2i} \alpha^i x^{2i}}}.$$

Using R\'{e}dei function in odd characteristic defined as above, we have the following construction of permutation polynomials over $\mathbb{F}_{q^2}$. \\

\begin{theorem}\cite{wang}
Let $q$ be an odd prime power.
Suppose $n > 0$ and $m$ are two integers.
Let $\alpha \in \mathbb{F}_{q^2}$ satisfy $\alpha^{q+1} =1$.
Then, the polynomial $$P(x)= x^{n +m(q+1)}H_n(x^{q-1},\alpha)$$
permutes $\mathbb{F}_{q^2}$ if and only if any one of the following conditions holds:
\begin{itemize}
    \item $\mathrm{gcd}(n(n+ 2m), q-1) =1$, when $\sqrt{\alpha} \in \mu_{q+1}$;
    
    \item $\mathrm{gcd}(n +2m, q-1)=1$ and $\mathrm{gcd}(n, q+1) =1$, when $\sqrt{\alpha} \notin \mu_{q+1}$.
\end{itemize}
\end{theorem}

We observe that the theorem above is also applicable to the polynomial $G_n(x, \alpha)$.
It is easy to show that we also have permutation polynomials in $\mathbb{F}_{q^2}$ using the denominator and the numerator of the tangent-Chebyshev function by adapting the proof of \cite[Theorem 1]{wang} with some simple algebraic manipulations as in \cite{zieve} and use of the following fundamental equality:
\begin{equation}\label{z}
    C_n(x, \alpha)= \frac{1}{x} \circ R_n(x, \alpha) \circ \frac{1}{x}.
\end{equation}

Equation (\ref{z}) implies the polynomial $F_n(x, \alpha)$, respectively $E_n(x, \alpha)$, is equivalent to the polynomial $G_n(x, \alpha)$, respectively $H_n(x, \alpha)$, by the equality $F_n(x, \alpha)= G_n(\frac{1}{x}, \alpha)$.
Therefore, we have the following corollary from \cite[Theorem 1]{wang}.\\

\begin{corollary}
    Let $q$ be an odd prime power, $\alpha \in \mathbb{F}_{q^2}$ with $\alpha^{q+1} =1$ and $n, m$ be positive integers.
    Then,
    $$P(x)= x^{n +m(q+1)}F_n(x^{q-1}, \alpha)$$
    permutes $\mathbb{F}_{q^2}$ if and only if one of the following condition holds:
    \begin{itemize}
        \item $\mathrm{gcd}(n(n +2m), q-1) =1$, when $\sqrt{\alpha} \in \mu_{q+1}$;
        \item  $\mathrm{gcd}(n+2m, q-1) =1$ and $\mathrm{gcd}(n, q+1) =1$, when $\sqrt{\alpha} \notin \mu_{q+1}$.
    \end{itemize}
\end{corollary}
We note that this corollary is also applicable to the polynomial $E_n(x, \alpha)$ as we have $E_n(x, \alpha)= H_n(\frac{1}{x}, \alpha)$.

Next, we define R\'{e}dei function in even characteristic.
Let $q$ be an even prime power and $h(x)= x^2 +x + \alpha$ be an irreducible polynomial over $\mathbb{F}_q$ where $\alpha \in \mathbb{F}_q$.
Let $\beta$ be a root of this polynomial in $\mathbb{F}_{q^2}$.
The other root of this polynomial is $\beta +1$ and therefore $\alpha = \beta^2 + \beta$ which implies $\beta^2 +\beta \in \mathbb{F}_q$ where $\beta, \beta +1 \in \mathbb{F}_{q^2} \setminus \mathbb{F}_q$.

We can define R\'{e}dei function in every non-zero characteristic finite field in the following way.
Here, for simplicity, $\Bar{\beta}= \beta + 1$ indicates the other root of $h(x)$ different from $\beta$ where $h(x)$ is an irreducible polynomial over $\mathbb{F}_q$ of the form $x^2 + x+ \alpha$ where $q$ is even.
If $q$ is odd, $\Bar{\beta} = -\beta$ indicates the root of $x^2 -\alpha$ different from $\beta$.\\

\begin{definition}\cite{zieve}\label{ziv}
The R\'{e}dei function of degree $n$ over $\mathbb{F}_q$ where $n$ is a positive integer and $\alpha \in \mathbb{F}_q$ is defined by
$$R_n(x,\alpha):= \rho^{-1} \circ x^n \circ \rho$$
where $\rho(x) := (x- \Bar{\beta}) / (x - \beta)$ and $\rho^{-1}(x) := (\beta x - \Bar{\beta}) / (x-1)$ are the degree-one rational functions in $\mathbb{F}_{q^2}(x)$ such that $(\rho^{-1} \circ \rho)(x) =x$.\\
\end{definition}

As we are interested in R\'{e}dei function in a field of even characteristic in this paper, by the definition above, we have the following :
\begin{equation}\label{redei}
    R_n(x, \alpha)= \frac{\beta(x + \beta+1)^n + (\beta+1)(x+ \beta)^n}{(x + \beta +1)^n +(x+ \beta)^n}.
\end{equation}  
We have the fact that $\rho(x)$ induces a bijection from $\mathbb{F}_q \cup \{ \infty\}$ to the set $\mu_{q+1}$ and it implies that $R_n(x, \alpha)$ permutes $\mathbb{F}_q \cup \{ \infty\} $ if and only if $x^n$ permutes $\mu_{q+ 1}$, i.e. $\gcd(n, q+1)=1$.

Let us consider the following notation:
$$M_n(x, \alpha) = (x + \beta +1)^n +(x+ \beta)^n = \sum_{i=0}^{n}{(\beta^i +(\beta +1)^i){n \choose i} x^{n-i}}$$
and
$$N_n(x, \alpha) = \beta(x + \beta+1)^n + (\beta+1)(x+ \beta)^n=  \sum_{i=0}^{n}{((\beta +1)\beta^i +\beta (\beta +1)^i){n \choose i} x^{n-i}}.$$

With this notation, the R\'{e}dei function in a field of characteristic 2 is defined as 
\begin{align*}
    R_n(x, \alpha)= \frac{N_n(x, \alpha)}{M_n(x, \alpha)} &= \frac{\beta(x + \beta+1)^n + (\beta+1)(x+ \beta)^n}{(x + \beta +1)^n +(x+ \beta)^n} \\
    &= \frac{ \sum_{i=0}^{n}{((\beta +1)\beta^i +\beta (\beta +1)^i){n \choose i} x^{n-i}}}{\sum_{i=0}^{n}{(\beta^i +(\beta +1)^i){n \choose i} x^{n-i}}}.
\end{align*}
We remark that we denote the numerator and denominator as $M_n(x,\alpha)$ and $N_n(x, \alpha)$ instead of $G_n(x, \alpha)$, $H_n(x, \alpha)$ to emphasize they are polynomials defined over a field of characteristic 2.

\section{A Construction of Permutation Polynomials in Even Characteristic}\label{sec3}

In this section, we introduce a recursive construction of two classes of permutation polynomials in a field of characteristic 2.
For this purpose, we  make some observations which are used in the proof of our main theorems.\\

Let $\beta \in \mathbb{F}_{q^2} \setminus \mathbb{F}_q$ such that $\alpha =\beta^2 + \beta =1$ where $q=2^t$, for some positive odd integer $t$. We claim $\beta^i +(\beta +1)^i \in \mathbb{F}_2$ and $(\beta + 1)\beta^i +\beta(\beta + 1)^i  \in \mathbb{F}_2$ for all positive integer $i$. Indeed, we observe that showing $\beta^i +(\beta + 1)^i \in \mathbb{F}_2$ is enough to show $(\beta + 1)\beta^i  + \beta (\beta +1)^i \in \mathbb{F}_2$ for all positive integer $i$, because we have the following relation by using the equalities $\beta = \beta^2 +1 =(\beta + 1)^2$ and $\beta^2= \beta + 1$.
  \begin{align*}
      (\beta + 1)\beta^i  + \beta (\beta +1)^i &=
      (\beta +1)(\beta +1)^{2i} + \beta\beta^{2i}  \\
      &= (\beta + 1)^{2i +1} + \beta^{2i +1}.    
  \end{align*}

 Similarly, we have
$ \beta^i + (\beta +1)^i = \beta^i + \beta^{2i} = (\beta + 1)^{2i} + \beta^{2i} 
      = ((\beta + 1)^i + \beta^i)^2$.
  That is $\beta^i + (\beta +1)^i \in \mathbb{F}_2$ for all positive integer $i$.\\

\begin{remark}
  Let $\beta \in \mathbb{F}_{q^2} \setminus \mathbb{F}_q$ and $\beta^2 + \beta =1$, i.e., $\alpha =1$.
  In other words, $\beta$ belongs to the extension of $\mathbb{F}_q$ with the minimal polynomial $x^2 +x + 1$.
  Suppose that $t$ is even, that is $q=2^{2k}$ for some positive integer $k$. 
  It means that $\mathbb{F}_q$ contains $\mathbb{F}_4$ as a subfield.
  Then, we can not consider the polynomial $x^2 +x+ \alpha$ where $\alpha =\beta^2 + \beta = 1$ to extend $\mathbb{F}_q$ to $\mathbb{F}_{q^2}$.
  Therefore, if $\beta^2 +\beta =1$, then $t$ should be odd where $q= 2^t$. \\
\end{remark}

From now on, we consider $q=2^t$ where $t$ is an odd integer and we fix $\alpha =1$.
We also simplify the notation and use $M_n(x)$ and $N_n(x)$ instead of $M_n(x,1)$ and $N_n(x,1)$.

We denote $a_i=\beta^i + (\beta +1)^i$ and $b_i = (\beta + 1)\beta^i  + \beta (\beta +1)^i$, we have
\begin{align*}
    b_i^2 &= ((\beta + 1)\beta^i  + \beta (\beta +1)^i)^2 = (\beta + 1)^2\beta^{2i}   + \beta^2 (\beta +1)^{2i} \\
    &= (\beta+1)^2(\beta + 1)^i  + \beta^2 \beta^i  =(\beta +1)^{i+ 2} + \beta^{i+2} \\
    &= a_{i+2}.
\end{align*}
Since $b_i \in \mathbb{F}_2$ when $\alpha =1$, this implies $(b_i)^2 = b_i$, and so we have the relation $b_i =a_{i+2}$.
Therefore, we have 
$$M_n(x) = \sum_{i=0}^{n}{a_i {n \choose i} x^{n-i}}~~\mbox{and}~~~N_n(x) = \sum_{i=0}^{n}{a_{i+2}{n \choose i} x^{n-i}}.$$

For some positive integer $k$, we also have the following equalities:
$$\beta^{3k}= (\beta \beta^2)^k = (\beta(\beta +1))^k= 1,$$
and
$$(\beta +1)^{3k} = ((\beta +1)(\beta +1)^2)^k = ((\beta+ 1) \beta)^k =1.$$
Therefore,
\begin{align*}
    a_{3k} &= \beta^{3k} +(\beta +1)^{3k}= 1+1 =0, \\
    a_{3k +1} &= \beta^{3k +1} +(\beta +1)^{3k +1} = \beta +(\beta +1) =1, \\
    a_{3k +2} &= \beta^{3k +2} +(\beta +1)^{3k +2} = \beta^2 +(\beta +1)^2 =1.
\end{align*}
We deduce that
\[\begin{cases} 
      a_i =0 & when \; 3 \mid i ,\\
      a_i=1 & otherwise.\\
\end{cases} \]

Determining $a_i$ and $b_i=a_{i+2}$ helps us to determine $M_n(x)$ and $N_n(x)$ easily, as we only need to calculate the corresponding binomial coefficient ${n \choose i} \pmod{2}$ to calculate $M_n(x)$ and $N_n(x)$.

It is easy to see that we have the following equalities
$$(x+ \beta)^n = N_n(x) +\beta M_n(x)$$
and 
$$(x +\beta +1)^n= N_n(x) +(\beta +1)M_n(x).$$

Similar to Lemma~2 in \cite{wang}, we obtain 

\begin{equation}\label{coprime}
\gcd(M_n(x), N_n(x)) =1. 
\end{equation}

In the following, we prove two lemmas required for the proof of our main theorems.\\

\begin{lemma}\label{noroot}
Let $q=2^t$ such that $t$ is an odd positive integer. Then
\begin{itemize}
    \item[(1)] $M_n(x)$ has no root in $\mu_{q+1}$ if $3\nmid n$. 
    \item[(2)] $N_n(x)$ has no root in $\mu_{q+1}$. 
\end{itemize}
\end{lemma}
\begin{proof}

We prove the results using three cases of $n \pmod 3$. 

\begin{itemize}
    \item[(1)]  Let $n \equiv 0 \pmod{3}$.
    Using $q\equiv 2 \pmod{3}$ and $\beta^3 = (\beta+1)^3=1$, we obtain $M_n(x)^q = x^{-n} M_n(x)$ and $$N_n(x)^q = x^{-n} N_n(x) + x^{-n} M_n(x)$$   for any $x \in \mu_{q+1}$.
    Suppose by contradiction that there is $x \in \mu_{q+1}$ such that $N_n(x)=0$.
    If $N_n(x) =0$  then $M_n(x) = 0$ as well.
    This contradicts to Equation~(\ref{coprime}), $N_n(x)$ has no root in $\mu_{q+1}$ .

    \item[(2)]      Let $n \equiv 1 \pmod{3}$.  
    Using $q\equiv 2 \pmod{3}$ and $\beta^3 = (\beta+1)^3=1$, we obtain  
    $$ N_n(x)^q =  \beta^2 (x^{-1} + (\beta+1)^2)^n + (\beta+1)^2 (x^{-1} + \beta^2)^n =  x^{-n} M_n(x) $$ 
    for any $x\in \mu_{q+1}$.
    If there is  $x \in \mu_{q+1}$ such that $N_n(x)=0$ or $M_n(x)=0$, then we obtain $N_n(x)=0$ or $M_n(x)=0$, respectively. 
    This contradicts to Equation~(\ref{coprime}), $M_n(x)$ and $N_n(x)$ has no root in $\mu_{q+1}$.

    \item[(3)]     Let $n \equiv 2 \pmod{3}$.  
    Using $q\equiv 2 \pmod{3}$ and $\beta^3 = (\beta+1)^3=1$, we obtain 
    $$ M_n(x)^q =  (x^{-1} + (\beta+1)^2)^n + (x^{-1} + \beta^2)^n =  x^{-n} N_n(x)$$
    for any $x\in \mu_{q+1}$.
     Similar to the previous case, if there is $x \in \mu_{q+1}$ such that $M_n(x)=0$ or $N_n(x) =0$, then we obtain $N_n(x)=0$ or $M_n(x)=0$, respectively. 
     This contradicts to Equation~(\ref{coprime}), $M_n(x)$ and $N_n(x)$ has no root in $\mu_{q+1}$.
    \end{itemize}
\end{proof}

 \begin{lemma}\label{rn}
 Let $q=2^t$ such that $t$ is an odd positive integer. 
 We have 
   \begin{itemize}
   
       \item[(1)]       $x^{n +m(q+1)}M_n(x)^{q-1} =1$ or $0$, when $n \equiv 0 \pmod{3}$ where $x \in \mu_{q+1}$;
       \item[(2)] $x^{n +m(q+1)}M_n(x)^{q-1} =R_n(x)$, when $n \equiv 1 \pmod{3}$ where $x \in \mu_{q+1}$;
       \item[(3)] $x^{n +m(q+1)}M_n(x)^{q-1} = R_n(x) +1$, when $n \equiv 2 \pmod{3}$ where $x \in \mu_{q+1}$;

       \item[(4)] $x^{n +m(q+1)}N_n(x)^{q-1} =1 + \frac{1}{R_n(x)}$, when $n \equiv 0 \pmod{3}$ where $x \in \mu_{q+1}$;
       \item[(5)] $x^{n +m(q+1)}N_n(x)^{q-1} = \frac{1}{R_n(x)}$, when $n \equiv 1 \pmod{3}$ where $x \in \mu_{q+1}$.     
          \item[(6)] $x^{n +m(q+1)}N_n(x)^{q-1} = 1$, when $n \equiv 2 \pmod{3}$ where $x \in \mu_{q+1}$.     
   \end{itemize}  
 \end{lemma}

\begin{proof}
    Consider $x^{n +m(q+1)}M_n(x)^{q-1}$ where $x \in \mu_{q+1}$.
    We showed in Lemma \ref{noroot} that $M_n(x)$ has no root in $\mu_{q+1}$ if $3\nmid n$.  In the case $3\mid n$, if $M_n(x) =0$ for some $x\in \mu_{q+1}$, then 
    $x^{n +m(q+1)}M_n(x)^{q-1} =0$.  From now on,  we assume $M_n(x) \neq 0$.  
    Since $\beta^q = \beta^{-1} = \beta+ 1$ and $(\beta+ 1)^q= (\beta +1)^{-1}= \beta$ where $\beta \in \mathbb{F}_{q^2} \setminus \mathbb{F}_q$, we have
    \begin{align*}
        x^{n +m(q+1)}M_n(x)^{q-1} &= x^{n +m(q+1)} \dfrac{((x+ \beta)^n + (x+ \beta +1)^n)^q}{(x+ \beta)^n + (x+ \beta +1)^n} \\
        &= x^n \dfrac{(x^q+ \beta^q)^n + (x^q+ (\beta +1)^q)^n}{(x+ \beta)^n + (x+ \beta +1)^n} \\
        &= \dfrac{((\beta+1)x+ 1)^n + (\beta x +1)^n}{(x+ \beta)^n + (x+ \beta +1)^n} \\
        &= \dfrac{\beta^n(x+ \beta +1)^n + (\beta +1)^n(x+ \beta)^n}{(x+ \beta)^n +(x+ \beta +1)^n}.\\
    \end{align*}

When $n \equiv 0 \pmod{3}$,   we have $\beta^n=1$ and $(\beta +1)^n= 1$. 
Therefore, we obtain
$$\dfrac{(x+ \beta +1)^n + (x + \beta)^n}{(x +\beta)^n + (x+ \beta +1)^n} = 1.$$

When $n \equiv 1 \pmod{3}$, we have $\beta^n= \beta$ and $(\beta +1)^n= \beta +1$.
Therefore, we obtain
$$\dfrac{\beta(x+ \beta +1)^n +(\beta +1)(x + \beta)^n}{(x +\beta)^n + (x+ \beta +1)^n} = R_n(x).$$

When $n \equiv 2 \pmod{3}$, we have $\beta^n = \beta^2 =\beta +1$ and $(\beta +1)^n= (\beta+ 1)^2= \beta$.
Then, the expression becomes
$$\dfrac{(\beta +1)(x+ \beta +1)^n +\beta (x + \beta)^n}{(x +\beta)^n + (x+ \beta +1)^n} = R_n(x) +1.$$
This proves $\textit{(1)}$-$\textit{(3)}$.

Consider $x^{n +m(q+1)}N_n(x)^{q-1}$ where $x \in \mu_{q+1}$.
By Lemma \ref{noroot}, we know that $N_n(x)$ has no root in $\mu_{q+1}$. 
Then we compute
\begin{align*}
    x^{n +m(q+1)}N_n(x)^{q-1} &= x^{n +m(q+1)} \dfrac{(\beta(x+ \beta +1)^n +(\beta+1)(x+ \beta)^n)^q}{\beta(x+ \beta +1)^n +(\beta+1)(x+ \beta)^n} \\
    &= x^n \dfrac{\beta^q(x^q+ (\beta +1)^q)^n +(\beta+1)^q(x^q+ \beta^q)^n}{\beta(x+ \beta +1)^n +(\beta+1)(x+ \beta)^n}\\
    &= \dfrac{(\beta +1)(\beta x +1)^n +\beta((\beta+1)x +1)^n}{\beta(x+ \beta +1)^n +(\beta+1)(x+ \beta)^n}\\
     &= \dfrac{(\beta +1)\beta^n(x +\beta +1)^n +\beta(\beta+1)^n(x +\beta)^n}{\beta(x+ \beta +1)^n +(\beta+1)(x+ \beta)^n}.\\
\end{align*}

When $n \equiv 0 \pmod{3}$, we have $\beta^n =(\beta+1)^n =1$, and we get
$$\dfrac{(\beta +1)(x +\beta +1)^n +\beta(x +\beta)^n}{\beta(x+ \beta +1)^n +(\beta+1)(x+ \beta)^n} = 1+ \frac{1}{R_n(x)}.$$

When $n \equiv 1 \pmod{3}$, we have $\beta^n = \beta$ and $(\beta+1)^n =\beta +1$, then we obtain
$$\dfrac{(x +\beta +1)^n +(x +\beta)^n}{\beta(x+ \beta +1)^n +(\beta+1)(x+ \beta)^n} = \frac{1}{R_n(x)}.$$

When $n \equiv 2 \pmod{3}$, we have $\beta^n = \beta^2 =\beta +1$ and $(\beta +1)^n= (\beta+ 1)^2= \beta$. Therefore
$$\dfrac{(\beta +1)\beta^n(x +\beta +1)^n +\beta(\beta+1)^n(x +\beta)^n}{\beta(x+ \beta +1)^n +(\beta+1)(x+ \beta)^n}=1. $$

This proves $\textit{(4)}$-$\textit{(6)}$.
\end{proof}

Next, we state and prove one of our main theorems of the paper.

\begin{theorem}\label{mainthm}
Let $\beta \in \mathbb{F}_{q^2} \setminus \mathbb{F}_q$ such that $\beta$ and $\beta +1$ are roots of the polynomial $x^2 +x +1$ in $\mathbb{F}_{q^2}$ where $q=2^t$ with $t$ odd and $n$, $m$  positive integers.
Consider $a_i=0$ if $3 \mid i$ and $a_i =1$ otherwise.
Then the polynomial
$$x^{n +m(q+1)}M_n(x^{q-1}) = x^{n+m(q+1)}\sum_{i=0}^{n}{a_i {n \choose i} (x^{q-1})^{n-i}}$$ 
permutes $\mathbb{F}_{q^2}$ if and only if $\mathrm{gcd}(n(n+ 2m), q-1) =1$ and  $3\nmid n$.
\end{theorem}
\begin{proof}
    We have the following diagram by Lemma \ref{agw} (AGW criterion):

    \begin{center}
    \begin{tikzcd}
    & & & \\
    \mathbb{F}_{q^2}^* \arrow{d}{x^{q-1}}  \arrow{rrr}{x^{n +m(q+1)}M_n(x^{q-1})}  & & &  \mathbb{F}_{q^2}^* \arrow{d}{x^{q-1}} \\
    \mu_{q+1} \arrow{rrr}{x^{n +m(q+1)}M_n(x)^{q-1}}  & & & \mu_{q+1} \\
    \end{tikzcd}
    \end{center}

    This entails that our result is equivalent to showing that $x^{n +m(q+1)}M_n(x)^{q-1}$ permutes $\mu_{q+1}$ when $\mathrm{gcd}(n +m(q+1),q-1)=1$ (i.e., $\mathrm{gcd}(n +2m,q-1)=1$).

By Lemma~\ref{rn},    we exclude the case $n \equiv 0 \pmod{3}$.   Hence we only need to examine two cases: $n \equiv 1 \pmod{3}$ and $n \equiv 2 \pmod{3}$.

\par{Case 1:} Let $n \equiv 1 \pmod{3}$.
    We know by Lemma \ref{rn} that $x^{n +m(q+1)}M_n(x)^{q-1} =R_n(x)$ when $n \equiv 1 \pmod{3}$, where $x \in \mu_{q+1}$.
    Therefore, we need to show that $R_n(x)$ permutes $\mu_{q+1}$.

Consider $\phi(x): \mathbb{F}_q \cup \{ \infty\}\rightarrow \mu_{q+1}$ such that $\phi(x) = \frac{x + \beta}{x + \beta +1}$ and $\phi^{-1}(x): \mu_{q+ 1} \rightarrow \mathbb{F}_q \cup \{ \infty\}$ such that $\phi^{-1}(x)= \frac{(\beta +1)x + \beta}{x+ 1}$.
We have the fact that $\phi(x)$ induces a bijection from $\mathbb{F}_q \cup \{ \infty\}$ to $\mu_{q+ 1}$ with $\phi(\infty)= 1$ and $\phi^{-1}(x)$ induces a bijection from $\mu_{q+ 1}$ to $\mathbb{F}_q \cup \{ \infty\}$ with $\phi^{-1}(1) =\infty$ as we have $(\phi \circ \phi^{-1})(x) =x$.
Therefore, $R_n(x)$ is a bijection on $\mu_{q+ 1}$ if and only if $R_n(\phi(x)): \mathbb{F}_q \cup \{ \infty \} \rightarrow \mu_{q+1} $ is a bijection if and only if $\phi^{-1}(R_n(\phi(x))): \mathbb{F}_q \cup \{ \infty\}\rightarrow \mathbb{F}_q \cup \{ \infty \}$ is a bijection.
To illustrate this, we have the following diagram:
\begin{center}
\begin{tikzcd}
 & & & \\
\mathbb{F}_{q^2}^* \arrow{d}{x^{q-1}}  \arrow{rrr}{x^{n+ m(q+1)}M_n(x^{q-1})}  & & & \mathbb{F}_{q^2}^* \arrow{d}{x^{q-1}}  \\
\mu_{q+1} \arrow{d}{\phi^{-1}(x)} \arrow{rrr}{R_n(x)}    & & & \mu_{q+1} \\
\mathbb{F}_q \cup\{ \infty \} \arrow{rrr}{ \phi^{-1}(R_n(\phi(x)))} & & & \arrow{u}{\phi(x)}   \mathbb{F}_q \cup\{ \infty \} 
\end{tikzcd}
\end{center}
We know that $R_n(x)$ maps 1 to 1, since we have
    $$ R_n(1) = \dfrac{\beta ^{n+1} +(\beta +1)^{n+1}}{\beta^n + (\beta +1)^n} =\dfrac{a_{n+1}}{a_n}.$$

Since $a_{n+1} =a_n= 1$ when $n \equiv 1 \pmod{3}$, we have $R_n(1)=1$.
Therefore, the function $\phi^{-1}(R_n(\phi(x)))$ on $\mathbb{F}_q \cup \{\infty \}$ maps $\infty$ to $\infty$.
Then, we need to show that it permutes $\mathbb{F}_q$.
One can compute $\phi^{-1}(R_n(\phi(x)))$ as follows using the equalities $\beta^n = \beta$ and $(\beta+ 1)^n= \beta+1$ when $n \equiv 1 \pmod{3}$.

\begin{align*}
    \phi^{-1}(R_n(\phi(x))) &= \dfrac{(\beta +1)R_n \Big( \frac{x+ \beta}{x+ \beta +1} \Big) + \beta}{R_n \Big( \frac{x+ \beta}{x+ \beta +1} \Big) +1} \\
    &= \dfrac{(\beta +1)(\beta x)^n}{(\beta +1)(\beta x)^n + \beta((\beta +1)x + \beta +1)^n} \\
    &= \dfrac{x^n}{x^n +(x+ 1)^n}.
\end{align*}

We denote this function as $g(x)$; next we show that $g(x)$ permutes $\mathbb{F}_q$. 
Suppose that $g(x) =g(y)$ with $x \neq y$ where $x,y \in \mathbb{F}_q$.
We have
$$\dfrac{x^n}{(x+1)^n + x^n} = \dfrac{y^n}{(y+ 1)^n +y^n}.$$
We have
$$x^n((y+1)^n +y^n) = y^n((x+ 1)^n +x^n),$$
that is,
$$(x(y+1))^n = ((x+ 1)y)^n.$$
We know that $\gcd(n, q-1)=1$. 
As a consequence, we have $xy +x = xy +y$ which implies $x =y$, a contradiction.
Therefore $g(x)$ permutes $\mathbb{F}_q$.

Thus, we showed that $R_n(x)$ permutes $\mu_{q+1} \setminus \{1 \}$ when $n \equiv 1 \pmod{3}$ and $R_n(1)= 1$, and as a consequence $R_n(x)$ permutes $\mu_{q+1}$.

\par{Case 2:} 
Let $n \equiv 2 \pmod{3}$.
We know by Lemma \ref{rn} that $x^{n+m(q+1)}M_n(x)^{q-1} =R_n(x) +1$, when $n \equiv 2 \pmod{3}$ where $x \in \mu_{q+1}$.
Therefore, we need to show that $R_n(x) +1$ permutes $\mu_{q+1}$.
Similar to the previous case, we need to show
$$\phi^{-1} \Big( \dfrac{(\beta +1)(\phi(x) + \beta +1)^n + \beta(\phi(x) +\beta)^n}{(\phi(x) + \beta)^n +(\phi(x) + \beta+1)^n} \Big)$$
permutes $\mathbb{F}_q \cup \{ \infty \}$.
To illustrate this, we have the following diagram:
\begin{center}
\begin{tikzcd}
 & & & \\
\mathbb{F}_{q^2}^* \arrow{d}{x^{q-1}}  \arrow{rrr}{x^{n+ m(q+1)}M_n(x^{q-1})}  & & & \mathbb{F}_{q^2}^* \arrow{d}{x^{q-1}}  \\
\mu_{q+1} \arrow{d}{\phi^{-1}(x)} \arrow{rrr}{R_n(x) +1}  & & &\mu_{q+1}  \\
\mathbb{F}_q \cup\{ \infty \} \arrow{rrr}{\phi^{-1} \circ (R_n(x) +1) \circ \phi} & & &  \arrow{u}{ \phi(x)}  \mathbb{F}_q \cup\{ \infty \} 
\end{tikzcd}
\end{center}
We observe that $R_n(x) +1$ maps 1 to 1.
Indeed, since $a_n= 1$ and $a_{n+1}= 0$ when $n \equiv 2 \pmod{3}$, we obtain 
$$R_n(1) + 1 =\frac{a_{n+1}}{a_n} +1 = 0+1 =1.$$
Therefore, the function $\phi^{-1} \circ ( R_n(x) +1)\circ \phi$ on $\mathbb{F}_q \cup \{ \infty\}$ maps $\infty$ to $\infty$.
Then, we need to show that it permutes $\mathbb{F}_q$.
One can compute  $\phi^{-1} \circ ( R_n(x) +1)\circ \phi$ as follows using the equalities $\beta^n = \beta^2 =\beta+1$ and $(\beta +1)^n= (\beta +1)^2 = \beta$ when $n \equiv 2 \pmod{3}$:
\begin{align*}
    \phi^{-1} \circ ( R_n(x) +1)\circ \phi &= 
    \phi^{-1} \Big( \dfrac{(\beta +1)(\phi(x) + \beta +1)^n + \beta(\phi(x) +\beta)^n}{(\phi(x) + \beta)^n +(\phi(x) + \beta+1)^n} \Big)\\
    &= \dfrac{(\beta +1)((\beta +1)x +\beta +1)^n}{\beta(\beta x)^n + (\beta +1)((\beta +1)x +\beta +1)^n }\\
    &= \dfrac{(x+1)^n}{x^n + (x+1)^n}.
\end{align*}

Let us denote this function by $g'(x)$, we want to show that $g' (x)$ permutes $\mathbb{F}_q$.
Suppose that $g'(x)= g'(y)$ with $x \neq y$ where $x,y \in \mathbb{F}_q$.
We have
$$\dfrac{(x+1)^n}{(x+ 1)^n +x^n} = \dfrac{(y+1)^n}{(y+1)^n + y^n},$$
that is,
$$(x+ 1)^n((y+1)^n +y^n) = (y+1)^n((x+1)^n +x^n).$$
This implies that 
$$(xy +y)^n =(xy +x)^n.$$
Since $\gcd(n,q-1)=1$, we have $xy +y = xy +x$ which implies $x=y$, a contradiction.
Therefore $R_n(x) +1$ is a permutation on $\mu_{q+1} \setminus \{ 1\}$.
Hence, $R_n(x) + 1$ permutes $\mu_{q+1}$ as $R_n(1) +1 =1$.

This shows $x^{n +m(q+1)}M_n(x)^{q-1}$ permutes $\mu_{q+1}$ when $3\nmid n$ and $\gcd(n, q -1)=1$,  which finishes the proof.
\end{proof}

Similarly, we can have permutation polynomials in a field of even characteristic by using the numerator $N_n(x)$ of the R\'{e}dei function in even characteristic.\\

\begin{theorem}\label{corr}
Let $\beta \in \mathbb{F}_{q^2} \setminus \mathbb{F}_q$ such that $\beta$ and $\beta +1$ are roots of the polynomial $x^2 +x +1$ in $\mathbb{F}_{q^2}$ where $q=2^t$ with $t$ odd and $n$, $m$ positive integers.
Consider $a_{i+2}=0$ if $3 \mid (i+2)$ and $a_{i+2}=1$ otherwise.
Then the polynomial
$$x^{n +m(q+1)}N_n(x^{q-1})= x^{n +m(q+1)}\sum_{i=0}^{n}{a_{i+2}{n \choose i} (x^{q-1})^{n-i}}$$ 
permutes $\mathbb{F}_{q^2}$ if and only if $\mathrm{gcd}(n(n+ 2m), q-1) =1$ and $n \not \equiv 2 \pmod{3}$.
\end{theorem}
\begin{proof}
    By Lemma \ref{rn}, $x^{n +m(q+1)}N_n(x^{q-1})$ becomes  $ 1+ \frac{1}{R_n(x)}$  or $ \frac{1}{R_n(x)}$ in $\mu_{q+1}$ when $n \not \equiv 2 \pmod{3}$.
    Similarly, we can show that $x^{n +m(q+1)}N_n(x^{q-1})$   permutes $\mu_{q+1}$ as in the proof of the previous theorem. 
\end{proof}

\begin{remark}
After we posted our results in arXiv, Ding and Zieve uploaded an arXiv paper motivated by our results. Although their paper contains some errors, they do have  a simple generalization of our Theorems~\ref{mainthm} and ~\ref{corr} to a class of polynomials of the form $a(x+u)^n + b(x+v)^n$. \\
\end{remark}

It is worth noticing that we can easily construct a permutation polynomial $x^{n+ m(q+1)}M_n(x^{q-1})$ or $x^{n+ m(q+1)}N_n(x^{q-1})$  over a field of even characteristic by only calculating the corresponding binomial coefficients ${n \choose i} \pmod{2}$ as we have 
$$M_n(x)= \sum_{i=0}^{n}{a_i {n \choose i} x^{n-i}}$$
and 
$$N_n(x)= \sum_{i=0}^{n}{a_{i+ 2} {n \choose i} x^{n-i}}$$
where $a_i= 0$ if $3 \mid i$ and $a_i=1$ otherwise.

Another useful property of R\'{e}dei function is that its numerator and denominator, $N_n(x)$, $M_n(x)$ respectively, can be obtained recursively.
Consider the R\'{e}dei function in even characteristic, $R_n(x)=\frac{N_n(x)}{M_n(x)}$.
It is easy to see that we have the following equalities
$$(x+ \beta)^n = N_n(x) +\beta M_n(x)$$
and 
$$(x +\beta +1)^n= N_n(x) +(\beta +1)M_n(x).$$
This allows us to generate $M_n(x)$ and $N_n(x)$ recursively,
\begin{align*}
    (x + \beta)^n &= (x + \beta)(x + \beta)^{n-1} \\
    &=(x+ \beta)(N_{n-1}(x) + \beta M_{n-1}(x)) \\ 
    &= xN_{n-1}(x) + \beta xM_{n-1}(x) + \beta N_{n-1}(x) +(\beta +1)M_{n-1}(x) \\
    &= \Big(xN_{n-1}(x) +M_{n-1}(x) \Big) +\beta \Big((x+1)M_{n-1}(x) +N_{n-1}(x) \Big).
\end{align*}
Equivalently, we have
\begin{align*}
    &(x+ \beta+ 1)^n = (x +\beta +1)(x+ \beta +1)^{n-1}\\
    &= xN_{n-1}(x) +(\beta +1)xM_{n-1}(x) +(\beta +1)N_{n-1}(x) +(\beta +1)M_{n-1}(x) + M_{n-1}(x) \\
    &= \Big( xN_{n-1}(x) + M_{n-1}(x) \Big) +(\beta +1) \Big( (x+1)M_{n-1}(x) + N_{n-1}(x) \Big).
\end{align*}

Therefore, we have the following recursive relation:
\begin{align*}
    M_0(x) &= 0,  \;\;\; N_0(x)=1, \\
    M_n(x) &=(x+1)M_{n-1}(x) + N_{n-1}(x), \\
    N_n(x) &= x N_{n-1}(x) + M_{n-1}(x).
\end{align*}

These properties allow us, by applying Theorem \ref{mainthm}, to obtain permutation polynomials over $\mathbb{F}_{q^2}$ where $q$ is an even prime power.
To obtain different permutation polynomials considering their reduction by $x^{q^2} +x$ using Theorem \ref{mainthm} and \ref{corr}, we need to consider $n \leq 3(q-1)$ and $m \leq q-1$ as the next theorem shows.\\

\begin{theorem}\label{rem}
    Let $n,m$ be positive integers such that $n \leq 3(q-1)$ and $m \leq q-1$.
    We have
    $$x^{n+3(q-1) +m(q+1)}M_{n+3(q-1)}(x^{q-1}) = x^{n+ m(q+1)}M_n(x^{q-1})$$
and 
$$x^{n+ m(q+1)}M_n(x) \equiv x^{n+ m'(q+1)}M_n(x) (\bmod{(x^{q^2} +x)})$$
where $m \equiv m' \pmod{(q-1)}$.
\end{theorem}

\begin{proof}
The left hand side of the first equality is
\begin{align*}
&x^{n+3(q-1) +m(q+1)}M_{n+3(q-1)}(x^{q-1}) \\
& = x^{n+3(q-1) +m(q+1)}((x^{q-1} +\beta)^{n +3(q-1)} + (x^{q-1} +\beta +1)^{n+ 3(q-1)}).
\end{align*}
Then, to show the first equality above, we need to show that
$$x^{3(q-1)}(x^{q-1} +\beta)^{3(q-1)} = x^{3(q-1)}(x^{q-1} +\beta +1)^{3(q-1)} =1.$$

We compute the following for $x \in \mathbb{F}_{q^2}^* \setminus \{1\}$:
\begin{align*}
    x^{3(q-1)}(x^{q-1} +\beta)^{3(q-1)} &= (x^q +\beta x)^{3(q-1)} \\
     &= (x^{3q} +\beta x^{2q+1} +(\beta +1)x^{q+2} +x^3)^{q-1}\\
    &= \frac{x^{3q^2} +\beta x^{(2q+1)q} +(\beta +1)x^{(q+2)q} +x^{3q}}{x^{3q} +\beta x^{2q+1} +(\beta +1)x^{q+2} +x^3} \\
    &= \frac{x^3 +(\beta +1)x^{q+2} +\beta x^{2q +1} +x^{3q}}{x^{3q} +\beta x^{2q+1} +(\beta +1)x^{q+2} +x^3} \\
    &= 1.
\end{align*}
We note that $x^{3(q-1)}(x^{q-1} +\beta)^{3(q-1)} =1$, when $x=1$, as $(\beta +1)^3 =1$.
Similarly, $x^{3(q-1)}(x^{q-1} +\beta +1)^{3(q-1)} =1$ for $x \in \mathbb{F}_{q^2}^*$.
Therefore, the equality $x^{3(q-1)}(x^{q-1} +\beta)^{3(q-1)} = x^{3(q-1)}(x^{q-1} +\beta +1)^{3(q-1)} =1$ holds.

For a positive integer $m < q-1$, it is straightforward that $x^{n +m' (q+ 1)} = x^{n+ m(q+1)}$ when $m' = m+s(q-1)$.
Thus, the second equality holds when $m \equiv m' \pmod{(q-1)}$.
\end{proof}

The tables below show permutation polynomials obtained from Theorem \ref{mainthm} and \ref{corr}.

\begin{center}
    
\begin{table}[h!]
\centering
\begin{tabular}{cc}

\begin{tabular}
{|p{1cm}|p{1cm}|p{2.2cm}|}

\hline

 \multirow{7}{4em}{$n=1$} & $m=1$ & $x^{10}$  \\
 & $m=2$& $x^{19}$  \\
  & $m=3$ & ---  \\
 & $m=4$ & $x^{37}$ \\
 &$m=5$ & $x^{46}$  \\
 &$m=6$& $x^{55}$  \\
 &$m=7$ &$x$ \\

    \hline

 \multirow{7}{4em}{$n=2$} & $m=1$ & $x^{11}$  \\
 & $m=2$& $x^{20}$  \\
 &$m=3$ &$x^{29}$ \\
 & $m=4$ & $x^{38}$ \\
 &$m=5$ & $x^{47}$ \\
 &$m=6$& --- \\
 &$m=7$ &$x^{2}$\\

\hline

  \multirow{7}{4em}{$n=4$} & $m=1$ & $x^{13}$  \\
 & $m=2$& $x^{22}$   \\
 &$m=3$ &$x^{31}$ \\
 & $m=4$ & $x^{40}$ \\
 &$m=5$ & --- \\
 &$m=6$& $x^{58}$ \\
 &$m=7$ &$x^{4}$ \\ 

 \hline

  \multirow{7}{4em}{$n=5$} & $m=1$ & ---  \\
 & $m=2$& $x^{51} +x^{30} +x^{23}$  \\
 &$m=3$ &$x^{60} +x^{39} +x^{32}$ \\
 & $m=4$ & $x^{48} +x^{41} +x^6$ \\
 &$m=5$ & $x^{57} +x^{50} +x^{15}$ \\
 &$m=6$& $x^{59} +x^{24} +x^3$ \\
 &$m=7$ &$x^{33} +x^{12} + x^5$ \\

\hline

 \multirow{7}{4em}{$n=8$} & $m=1$ & $x^{17}$  \\
 & $m=2$& $x^{26}$  \\
 &$m=3$ & ---  \\
 & $m=4$ & $x^{44}$  \\
 &$m=5$ & $x^{53}$ \\
 &$m=6$& $x^{62}$  \\
 &$m=7$ &$x^{8}$ \\

 \hline

  \multirow{7}{4em}{$n=10$} & $m=1$ & $x^{33} +x^{19} +x^{12}$ \\
 & $m=2$& ---  \\
 &$m=3$ &$x^{51} +x^{37} +x^{30}$ \\
 & $m=4$ & $x^{60} +x^{46} +x^{39}$  \\
 &$m=5$ & $x^{55} +x^{48} +x^{6}$ \\\
 &$m=6$& $x^{57} +x^{15} +x$ \\
 &$m=7$ &$x^{24} +x^{10} + x^3$ \\
\hline

\end{tabular} & 
\begin{tabular}{|p{1cm}|p{1cm}|p{2.2cm}|}
\hline
  \multirow{7}{4em}{$n=11$} & $m=1$ & $x^{41}$ \\
 & $m=2$& $x^{50}$ \\
 &$m=3$ & $x^{59}$ \\
 & $m=4$ & $x^{5}$ \\
 &$m=5$ & --- \\
 &$m=6$& $x^{23}$ \\
 &$m=7$ &$x^{32}$ \\

 \hline

  \multirow{7}{4em}{$n=13$} & $m=1$ & $x^{57} +x^{43} +x^{15}$  \\
 & $m=2$& $x^{52} +x^{24} +x^{3}$ \\
 &$m=3$ &$x^{62} +x^{33} +x^{12}$ \\
 & $m=4$ & --- \\
 &$m=5$ & $x^{51} +x^{30} +x^{16}$ \\
 &$m=6$& $x^{60} +x^{39} +x^{25}$ \\
 &$m=7$ &$x^{48} +x^{34} + x^6$ \\
\hline

  \multirow{7}{4em}{$n=16$} & $m=1$ & $x^{25}$  \\
 & $m=2$& $x^{34}$ \\
 &$m=3$ & $x^{43}$ \\
 & $m=4$ & $x^{52}$ \\
 &$m=5$ & $x^{61}$ \\
 &$m=6$& ---  \\
 &$m=7$ &$x^{16}$ \\

 \hline

   \multirow{7}{4em}{$n=17$} & $m=1$ & $x^{33} +x^{26} +x^{12}$  \\
 & $m=2$& --- \\
 &$m=3$ &$x^{51} +x^{44} +x^{30}$ \\
 & $m=4$ & $x^{60} +x^{53} +x^{39}$  \\
 &$m=5$ & $x^{62} +x^{48} +x^{6}$ \\
 &$m=6$& $x^{57} +x^{15} +x^{8}$ \\
 &$m=7$ &$x^{24} +x^{17} + x^3$ \\
\hline

   \multirow{7}{4em}{$n=19$} & $m=1$ & --- \\
 & $m=2$& $x^{58} +x^{51} +x^{30}$ \\
 &$m=3$ &$x^{60} +x^{39} +x^{4}$ \\
 & $m=4$ & $x^{48} +x^{13} +x^{6}$ \\
 &$m=5$ & $x^{57} +x^{22} +x^{15}$ \\
 &$m=6$& $x^{31} +x^{24} +x^{3}$ \\
 &$m=7$ &$x^{40} +x^{33} + x^{12}$ \\
\hline

    \multirow{7}{4em}{$n=20$} & $m=1$ & $x^{57} +x^{29} +x^{15}$ \\
 & $m=2$& $x^{38} +x^{24} +x^{3}$ \\
 &$m=3$ &$x^{47} +x^{33} +x^{12}$ \\
 & $m=4$ & ---\\
 &$m=5$ & $x^{51} +x^{30} +x^{2}$ \\
 &$m=6$& $x^{60} +x^{39} +x^{11}$ \\
 &$m=7$ &$x^{48} +x^{20} + x^{6}$ \\
\hline

\end{tabular} \\
\end{tabular}
\caption{Permutation polynomials of $\mathbb{F}_{2^6}$ obtained by applying Theorem \ref{mainthm}, and reduced by $x^{2^6} +x$.}
\label{table:1}
\end{table}

\end{center}

\begin{table}[h!]
\centering
\begin{tabular}{cc}
\begin{tabular}{|p{1cm}|p{1cm}|p{2.2cm}|}
\hline

 \multirow{7}{4em}{$n=1$} & $m=1$ & $x^{17}$ \rule[1ex]{0pt}{1.17ex} \\
 & $m=2$& $x^{26}$ \rule[1ex]{0pt}{1.17ex} \\
  & $m=3$ & --- \rule[1ex]{0pt}{1.17ex}\\
 & $m=4$ & $x^{44}$ \rule[1ex]{0pt}{1.17ex}\\
 &$m=5$ & $x^{53}$ \rule[1ex]{0pt}{1.17ex}\\
 &$m=6$& $x^{62}$ \rule[1ex]{0pt}{1.17ex}\\
 &$m=7$ &$x^8$ \rule[1ex]{0pt}{1.17ex}\\

    \hline
 \multirow{7}{4em}{$n=3$} & $m=1$ & $x^{33} +x^{19} +x^{12}$ \rule[1ex]{0pt}{1.17ex}  \\
 & $m=2$& ---  \rule[1ex]{0pt}{1.17ex} \\
 &$m=3$ &$x^{51} +x^{37} +x^{30}$ \rule[1ex]{0pt}{1.17ex} \\
 & $m=4$ & $x^{60} +x^{46} +x^{39}$ \rule[1ex]{0pt}{1.17ex} \\
 &$m=5$ & $x^{55} +x^{48} +x^{6}$ \rule[1ex]{0pt}{1.17ex} \\
 &$m=6$& $x^{57} +x^{15} +x$ \rule[1ex]{0pt}{1.17ex} \\
 &$m=7$ &$x^{24} +x^{10} + x^3$ \rule[1ex]{0pt}{1.17ex}\\

 \hline

 \multirow{7}{4em}{$n=4$} & $m=1$ & $x^{41}$ \rule[1ex]{0pt}{1.17ex} \\
 & $m=2$& $x^{50}$  \rule[1ex]{0pt}{1.17ex}\\
 &$m=3$ &$x^{59}$ \rule[1ex]{0pt}{1.17ex}\\
 & $m=4$ & $x^{5}$ \rule[1ex]{0pt}{1.17ex}\\
 &$m=5$ & --- \rule[1ex]{0pt}{1.17ex}\\
 &$m=6$& $x^{23}$ \rule[1ex]{0pt}{1.17ex}\\
 &$m=7$ &$x^{32}$ \rule[1ex]{0pt}{1.17ex}\\

\hline
 \multirow{7}{4em}{$n=6$} & $m=1$ & $x^{57} +x^{43} +x^{15}$ \rule[1ex]{0pt}{1.17ex}  \\
 & $m=2$& $x^{52} +x^{24} +x^{3}$  \rule[1ex]{0pt}{1.17ex} \\
 &$m=3$ &$x^{61} +x^{33} +x^{12}$ \rule[1ex]{0pt}{1.17ex} \\
 & $m=4$ & --- \rule[1ex]{0pt}{1.17ex} \\
 &$m=5$ & $x^{51} +x^{30} +x^{16}$ \rule[1ex]{0pt}{1.17ex} \\
 &$m=6$& $x^{60} +x^{39} +x^{25}$ \rule[1ex]{0pt}{1.17ex} \\
 &$m=7$ &$x^{48} +x^{34} + x^{6}$ \rule[1ex]{0pt}{1.17ex}\\

\hline
 \multirow{7}{4em}{$n=9$} & $m=1$ & $x^{25}$ \rule[1ex]{0pt}{1.17ex} \\
 & $m=2$& $x^{34}$  \rule[1ex]{0pt}{1.17ex}\\
 &$m=3$ &$x^{43}$ \rule[1ex]{0pt}{1.17ex}\\
 & $m=4$ & $x^{52}$ \rule[1ex]{0pt}{1.17ex}\\
 &$m=5$ & $x^{61}$ \rule[1ex]{0pt}{1.17ex}\\
 &$m=6$& --- \rule[1ex]{0pt}{1.17ex}\\
 &$m=7$ &$x^{16}$ \rule[1ex]{0pt}{1.17ex}\\
\hline
  \multirow{7}{4em}{$n=10$} & $m=1$ & $x^{33} +x^{26} +x^{12}$ \rule[1ex]{0pt}{1.17ex}  \\
 & $m=2$& ---  \rule[1ex]{0pt}{1.17ex} \\
 &$m=3$ &$x^{51} +x^{44} +x^{30}$ \rule[1ex]{0pt}{1.17ex} \\
 & $m=4$ & $x^{60} +x^{53} +x^{39}$ \rule[1ex]{0pt}{1.17ex} \\
 &$m=5$ & $x^{62} +x^{48} +x^{6}$ \rule[1ex]{0pt}{1.17ex} \\
 &$m=6$& $x^{57} +x^{15} +x^8$ \rule[1ex]{0pt}{1.17ex} \\
 &$m=7$ &$x^{24} +x^{17} + x^3$ \rule[1ex]{0pt}{1.17ex}\\
  \hline
\end{tabular} &

\begin{tabular}{|p{1cm}|p{1cm}|p{2.2cm}|}
\hline
  \multirow{7}{4em}{$n=12$} & $m=1$ & --- \rule[1ex]{0pt}{1.17ex} \\
 & $m=2$& $x^{58} +x^{51} +x^{30}$ \rule[1ex]{0pt}{1.17ex} \\
 &$m=3$ &$x^{60} +x^{39} +x^{4}$ \rule[1ex]{0pt}{1.17ex}\\
 & $m=4$ & $x^{48} +x^{13} +x^{6}$ \rule[1ex]{0pt}{1.17ex}\\
 &$m=5$ & $x^{57} +x^{22} +x^{15}$ \rule[1ex]{0pt}{1.17ex}\\
 &$m=6$& $x^{31} +x^{24} +x^{3}$ \rule[1ex]{0pt}{1.17ex}\\
 &$m=7$ &$x^{40} +x^{33} + x^{12}$ \rule[1ex]{0pt}{1.17ex}\\
\hline
  \multirow{7}{4em}{$n=13$} & $m=1$ & $x^{57} +x^{29} +x^{15}$ \rule[1ex]{0pt}{1.17ex} \\
 & $m=2$& $x^{38} +x^{24} +x^{3}$ \rule[1ex]{0pt}{1.17ex} \\
 &$m=3$ &$x^{47} +x^{33} +x^{12}$ \rule[1ex]{0pt}{1.17ex}\\
 & $m=4$ & --- \rule[1ex]{0pt}{1.17ex}\\
 &$m=5$ & $x^{51} +x^{30} +x^{2}$ \rule[1ex]{0pt}{1.17ex}\\
 &$m=6$& $x^{60} +x^{39} +x^{11}$ \rule[1ex]{0pt}{1.17ex}\\
 &$m=7$ &$x^{48} +x^{20} + x^6$ \rule[1ex]{0pt}{1.17ex}\\

\hline
 \multirow{7}{4em}{$n=15$} & $m=1$ & $x^{10}$  \rule[1ex]{0pt}{1.17ex}\\
 & $m=2$& $x^{19}$  \rule[1ex]{0pt}{1.17ex}\\
 &$m=3$ & --- \rule[1ex]{0pt}{1.17ex}\\
 & $m=4$ & $x^{37}$ \rule[1ex]{0pt}{1.17ex}\\
 &$m=5$ & $x^{46}$ \rule[1ex]{0pt}{1.17ex}\\
 &$m=6$& $x^{55}$ \rule[1ex]{0pt}{1.17ex}\\
 &$m=7$ &$x$ \rule[1ex]{0pt}{1.17ex}\\

\hline
 \multirow{7}{4em}{$n=16$} & $m=1$ & $x^{11}$  \rule[1ex]{0pt}{1.17ex}\\
 & $m=2$& $x^{20}$  \rule[1ex]{0pt}{1.17ex}\\
 &$m=3$ & $x^{29}$ \rule[1ex]{0pt}{1.17ex}\\
 & $m=4$ & $x^{38}$ \rule[1ex]{0pt}{1.17ex}\\
 &$m=5$ & $x^{47}$ \rule[1ex]{0pt}{1.17ex}\\
 &$m=6$& --- \rule[1ex]{0pt}{1.17ex}\\
 &$m=7$ &$x^{2}$ \rule[1ex]{0pt}{1.17ex}\\

 \hline
 \multirow{7}{4em}{$n=18$} & $m=1$ & $x^{13}$  \rule[1ex]{0pt}{1.17ex}\\
 & $m=2$& $x^{22}$  \rule[1ex]{0pt}{1.17ex}\\
 &$m=3$ & $x^{31}$ \rule[1ex]{0pt}{1.17ex}\\
 & $m=4$ & $x^{40}$ \rule[1ex]{0pt}{1.17ex}\\
 &$m=5$ & --- \rule[1ex]{0pt}{1.17ex}\\
 &$m=6$& $x^{58}$ \rule[1ex]{0pt}{1.17ex}\\
 &$m=7$ &$x^{4}$ \rule[1ex]{0pt}{1.17ex}\\

\hline
  \multirow{7}{4em}{$n=19$} & $m=1$ & ---  \rule[1ex]{0pt}{1.17ex}\\
 & $m=2$& $x^{51} +x^{30} +x^{23}$  \rule[1ex]{0pt}{1.17ex}\\
 &$m=3$ &$x^{60} +x^{39} +x^{32}$ \rule[1ex]{0pt}{1.17ex}\\
 & $m=4$ & $x^{48} +x^{41} +x^{6}$ \rule[1ex]{0pt}{1.17ex}\\
 &$m=5$ & $x^{57} +x^{50} +x^{15}$ \rule[1ex]{0pt}{1.17ex}\\
 &$m=6$& $x^{59} +x^{24} +x^3$ \rule[1ex]{0pt}{1.17ex} \\
 &$m=7$ &$x^{33} +x^{12} + x^5$ \rule[1ex]{0pt}{1.17ex}\\
\hline
\end{tabular} 
\end{tabular}
\caption{Permutation polynomials of $\mathbb{F}_{2^6}$ obtained by applying Theorem \ref{corr}, and reduced by $x^{2^6} +x$.}
\label{table:2}
\end{table}

The first two tables, Table \ref{table:1} and Table \ref{table:2}, contain permutation polynomials of $\mathbb{F}_{2^6}$.
According to Theorem \ref{rem} above, we consider $n\leq 21$ and $m \leq 7$ which satisfies conditions in Theorem \ref{mainthm} and \ref{corr} and these tables provide all permutation polynomials of $\mathbb{F}_{2^6}$ that can be obtained from the construction that we present.

\begin{table}[h!]
\centering
\begin{tabular}{|p{1cm}|p{1cm}|p{4.5cm}|}
\hline
\multirow{3}{4em}{$n=7$} & $m=1$ & $x^{226} +x^{195} +x^{133} +x^{102} +x^{40}$ \rule[1ex]{0pt}{1.17ex}\\
&$m=2$ & $x^{259} +x^{228} +x^{166} +x^{135} +x^{73}$ \rule[1ex]{0pt}{1.17ex}\\
&$m=3$ & $x^{292} +x^{261} +x^{199} +x^{168} +x^{106}$ \rule[1ex]{0pt}{1.17ex}\\

 \hline
\multirow{3}{4em}{$n=13$} & $m=1$ & $x^{418} +x^{325} +x^{294} +x^{201} +x^{46}$ \rule[1ex]{0pt}{1.17ex}\\
& $m=2$&  $x^{451} +x^{358} +x^{327} +x^{234} +x^{79}$ \rule[1ex]{0pt}{1.17ex}\\
&$m=3$ & $x^{484} +x^{391} +x^{360} +x^{267} +x^{112}$\rule[1ex]{0pt}{1.17ex} \\
 \hline
 
 \multirow{3}{4em}{$n=34$} & $m=1$ & $x^{129} +x^{67} +x^{36}$ \rule[1ex]{0pt}{1.17ex} \\
 & $m=2$ & $x^{162} +x^{100} +x^{69}$ \rule[1ex]{0pt}{1.17ex}\\
 & $m=3$ & $x^{195} +x^{133} + x^{102}$ \rule[1ex]{0pt}{1.17ex} \\
 \hline

\end{tabular} 
\caption{Some permutation polynomials obtained from Theorem \ref{mainthm} over $\mathbb{F}_{2^{10}}$, and reduced by $x^{2^{10}} +x$.}
\label{table:3}
\end{table}

\begin{table}[h!]
\centering
\begin{tabular}{|p{1cm}|p{1cm}|p{4.5cm}|}
\hline

 \multirow{3}{4em}{$n=7$} & $m=1$&  $x^{257} +x^{195} +x^{164} +x^{102} +x^{71}$ \rule[1ex]{0pt}{1.17ex}\\
 &$m=2$& $x^{290} +x^{228} +x^{197} +x^{135} +x^{104}$ \rule[1ex]{0pt}{1.17ex} \\
 &$m=3$ & $x^{323} +x^{261} +x^{230} +x^{168} +x^{137}$ \rule[1ex]{0pt}{1.17ex}\\

 \hline
 \multirow{3}{4em}{$n=13$}&$m=1$&  $x^{449} +x^{294} +x^{201} +x^{170} +x^{77}$ \rule[1ex]{0pt}{1.17ex}\\
 &$m=2$& $x^{482} +x^{327} +x^{234} +x^{203} +x^{110}$ \rule[1ex]{0pt}{1.17ex} \\
 &$m=3$& $x^{515} +x^{360} +x^{267} +x^{236} +x^{143}$ \rule[1ex]{0pt}{1.17ex}\\
 \hline
 
  \multirow{3}{4em}{$n=34$} & $m=1$& $x^{129} +x^{98} +x^{36}$ \rule[1ex]{0pt}{1.17ex}\\
  &$m=2$& $x^{162} +x^{131} +x^{69}$ \rule[1ex]{0pt}{1.17ex} \\
   &$m=3$ & $x^{195} +x^{164} + x^{102}$ \rule[1ex]{0pt}{1.17ex} \\

\hline
\end{tabular}
\caption{Some permutation polynomials obtained from Theorem \ref{corr} over $\mathbb{F}_{2^{10}}$, and reduced by $x^{2^{10}} +x$.}
\label{table:4}
\end{table}

\begin{table}[h!]
\centering
\begin{tabular}{|p{1cm}|p{1cm}|p{4.5cm}|}
\hline 
\multirow{3}{4em}{$n=11$}& $m=1$ & $x^{1410} +x^{1283} +x^{521} +x^{267} +x^{140}$ \rule[1ex]{0pt}{1.17ex}\\
&$m=2$ & $x^{1539} +x^{1412} +x^{650} +x^{396} +x^{269}$ \rule[1ex]{0pt}{1.17ex}\\
&$m=3$ & $x^{1668} +x^{1541} +x^{779} +x^{525} +x^{398}$ \rule[1ex]{0pt}{1.17ex}\\
 \hline
 \multirow{3}{4em}{$n=20$}& $m=1$& $x^{2181} +x^{657} +x^{149}$ \rule[1ex]{0pt}{1.17ex} \\
 &$m=2$& $x^{2310} +x^{786} +x^{278}$ \rule[1ex]{0pt}{1.17ex} \\
&$m=3$ & $x^{2439} +x^{915} +x^{407}$ \rule[1ex]{0pt}{1.17ex}\\
  \hline
  \multirow{3}{4em}{$n=56$}& $m=1$ & $x^{6281} +x^{5265} +x^{3233} +x^{2217} +x^{185}$ \rule[1ex]{0pt}{1.17ex}\\
  &$m=2$ & $x^{6410} +x^{5394} +x^{3362} +x^{2346} +x^{314}$ \rule[1ex]{0pt}{1.17ex} \\
  &$m=3$& $x^{6539} +x^{5523} +x^{3491} +x^{2475} +x^{443}$ \rule[1ex]{0pt}{1.17ex}\\
  \hline

\multirow{3}{4em}{$n=134$} & $m=1$ &$x^{1025} +x^{771} +x^{644} +x^{390} +x^{263}$ \rule[1ex]{0pt}{1.17ex}\\
&$m=2$ & $x^{1154} +x^{900} +x^{773} +x^{519} +x^{392}$ \rule[1ex]{0pt}{1.17ex}\\
&$m=3$ & $x^{1283} +x^{1029} + x^{902} +x^{648} +x^{521}$ \rule[1ex]{0pt}{1.17ex} \\

\hline
\end{tabular}
\caption{Some permutation polynomials obtained from Theorem \ref{mainthm} over $\mathbb{F}_{2^{14}}$, and reduced by $x^{2^{14}} +x$.}
\label{table:5}
\end{table}

\clearpage

\begin{table}[h!]
\centering
\begin{tabular}{|p{1cm}|p{1cm}|p{4.5cm}|}
\hline
 \multirow{3}{4em}{$n=7$}& $m=1$ & $x^{1025} +x^{771} +x^{644} +x^{390} +x^{263}$ \rule[1ex]{0pt}{1.17ex}\\
 &$m=2$ & $x^{1154} +x^{900} +x^{773} +x^{519} +x^{392}$ \rule[1ex]{0pt}{1.17ex}\\
 &$m=3$ & $x^{1283} +x^{1029} +x^{902} +x^{648} +x^{521}$ \rule[1ex]{0pt}{1.17ex} \\
\hline
  \multirow{3}{4em}{$n=22$} & $m=1$& $x^{2945} +x^{2691} +x^{2183} +x^{659} +x^{405}$ \rule[1ex]{0pt}{1.17ex}\\
  &$m=2$ & $x^{3074} +x^{2820} +x^{2312} +x^{788} +x^{534}$ \rule[1ex]{0pt}{1.17ex} \\
  &$m=3$ & $x^{3203} +x^{2949} +x^{2441} +x^{917} +x^{663}$ \rule[1ex]{0pt}{1.17ex} \\
  \hline
  \multirow{3}{4em}{$n=56$} & $m=1$& $x^{6281} +x^{5265} +x^{3233} +x^{2217} +x^{185}$ \rule[1ex]{0pt}{1.17ex}\\
  &$m=2$& $x^{6410} +x^{5394} +x^{3362} +x^{2346} +x^{314}$\rule[1ex]{0pt}{1.17ex} \\
  &$m=3$ & $x^{6539} +x^{5523} +x^{3491} +x^{2475} +x^{443}$ \rule[1ex]{0pt}{1.17ex}\\
  \hline

\multirow{3}{4em}{$n=136$} &$m=1$ & $x^{1281} +x^{1154} +x^{138}$ \rule[1ex]{0pt}{1.17ex} \\
&$m=2$ & $x^{1410} +x^{1283} +x^{267}$  \rule[1ex]{0pt}{1.17ex} \\
&$m=3$ & $x^{1539} +x^{1412} + x^{396}$ \rule[1ex]{0pt}{1.17ex} \\

 \hline
\end{tabular}
\caption{Some permutation polynomials obtained from Theorem \ref{corr} over $\mathbb{F}_{2^{14}}$, and reduced by $x^{2^{14}} +x$.}
\label{table:6}
\end{table}
Table \ref{table:3} and \ref{table:4}, present some permutation polynomials obtained by the same construction over $\mathbb{F}_{2^{10}}$.
Similarly, Table \ref{table:5} and \ref{table:6} provide some permutation polynomials of $\mathbb{F}_{2^{14}}$ according to Theorem \ref{mainthm} and \ref{corr}.

\section{Conclusion}\label{sec4}
In this paper, we provide a recursive construction of permutation polynomials by using the numerator and the denominator of the R\'{e}dei function in connection to the AGW criterion.
This construction provides a convenient way to obtain permutation polynomials in a field of even characteristic as it only requires to compute binomial coefficients modulo 2; the polynomials can be obtained recursively.
We give examples of permutation polynomials obtained by applying Theorem \ref{mainthm} and \ref{corr} in different finite fields.
In this article, we fix $\alpha = \beta^2 + \beta =1$ and the construction is given by considering the R\'{e}dei function in even characteristic under this assumption.

\section*{Acknowledgements}

The authors were supported by the Natural Sciences and Engineering Research Council of Canada, projects
	RGPIN-2018-05328 (D. Panario) and RGPIN-2023-04673 (Q. Wang).

\end{document}